\documentclass[review]{elsarticle}

\usepackage{lineno,hyperref}
\modulolinenumbers[5]

\usepackage{amsmath}
\usepackage{amsthm}
\usepackage{amssymb}
\usepackage{mathrsfs}
\usepackage{array}
\usepackage{tabu}
\usepackage[T1]{fontenc}

\usepackage[a4paper]{geometry}
\usepackage{amsfonts}
\usepackage{parskip}
\usepackage{graphicx}
\usepackage{xtab}
\usepackage{setspace}
\usepackage[nodayofweek]{datetime}
\usepackage{chngpage}

\theoremstyle{plain}
\newtheorem{theorem}{Theorem}[section]
\newtheorem{lemma}[theorem]{Lemma}
\newtheorem{proposition}[theorem]{Proposition}

\newtheorem{conjecture}[theorem]{Conjecture}

\theoremstyle{definition}

\theoremstyle{remark}

\newtheorem{example}[theorem]{Example}

\newdateformat{monthyear}{\monthname[\THEMONTH], \THEYEAR}

\journal{Journal of Algebra}

\begin{document}

\pagenumbering{roman}
\include{titlepage}
\setcounter{page}{2}
\include{contents}
\onehalfspacing
\pagenumbering{arabic}
\onehalfspacing

\newcommand{\B}[1]{\mathbb{#1}}
\newcommand{\BF}[1]{\mathbf{#1}}
\newcommand{\twopartdef}[4]{\left\{
		\begin{array}{ll}
			#1 & \mbox{if } #2 \\
			#3 & \mbox{if } #4
		\end{array}
	\right.}

\begin{frontmatter}

\title{$K_2$ of Kac-Moody Groups}

\author{Matthew Westaway}
\address{Mathematics Institute, University of Warwick, Coventry, CV4 7AL, United Kingdom}

\begin{abstract}
	Ulf Rehmann and Jun Morita, in their 1989 paper \emph{A Matsumoto Type Theorem for Kac-Moody Groups}, gave a presentation of 
	$K_2(A,F)$ for any generalised Cartan matrix $A$ and field $F$. The purpose of this paper is to use this presentation to compute $K_2(A,F)$ more explicitly in the case when $A$ is hyperbolic. In particular, we shall show that these $K_2(A,F)$ can always be expressed as a product of quotients of $K_2(F)$ and $K_2(2,F)$. Along the way, we shall also prove a similar result in the case when $A$ has an odd entry in each column.
\end{abstract}

\begin{keyword}
	Kac-Moody Groups \sep Matsumoto's Theorem \sep $K_2(A,F)$ \sep K-Theory
	\MSC[2010] 19C20\sep  17B67
\end{keyword}

\end{frontmatter}

\section{Introduction}

In 1971, John Milnor introduced the functor $K_2$, which assigns to each field $F$ an abelian group $K_2(F)$. The definition was straightforward: for $n\geq 3$, $K_2(n,F)$ was simply the kernel of the natural homomorphism from the Steinberg group $St(n,F)$ (cf. \cite{IAKT}) into $GL(n,F)$, and then $K_2(F)$ was defined as the direct limit of the $K_2(n,F)$. In fact, Milnor further proved that $K_2(n,F)$ was isomorphic to $K_2(F)$ for all $n\geq 3$.

While Milnor did not deal with the $n=2$ case, a small change to the definition of the Steinberg group in this case yields an analogous group $K_2(2,F)$ (cf. \cite{DVR}). So for each field $F$ we have two groups of interest: $K_2(F)$ and $K_2(2,F)$. In \cite{MAT}, Matsumoto was able to give presentations for both of these groups in terms of generators and relations, using ideas from root systems.

Once Tits introduced the idea of a Kac-Moody Group in \cite{UPKMG}, a generalisation of $K_2$ was born, one which was related to the root systems of arbitrary Kac-Moody algebras as opposed to just the finite ones which the original $K_2$ had been (cf. \S 6, \S 7 in \cite{SLOCG}, \S 1 in \cite{AMTTFKMG}). As a natural progression, Rehmann and Morita in \cite{AMTTFKMG} similarly generalised Matsumoto's seminal theorem in \cite{MAT} to give a presentation of this new $K_2(A,F)$ for \emph{all} Kac-Moody Lie algebras $A$ and fields $F$.

This presentation is similar to Matsumoto's, but due to the greater generality of the root systems involved it is inherently more complicated. As a result, what $K_2(A,F)$ actually is for a given Generalised Cartan Matrix (GCM) $A$ and field $F$ is more difficult to determine than we would ideally like. Hence, in this paper we aim to give an alternate presentation for $K_2(A,F)$ for the hyperbolic GCMs, which should be easier to work with. In particular, it reduces to understanding the structure of the groups $K_2(F)$ and $K_2(2,F)$, a subject in which a fair amount of work has already been done.

This paper shall start by introducing the key definitions, theorems and notation which shall be used throughout. Next, in Section 3, we will explicate some of aspects of Rehmann and Morita's paper in more detail, in order to give us some more tools to work with, and we will use these to derive some straightforward results in Section 4. In Section 5, we shall give a general result which holds for all GCMs of a certain form which appears very frequently for hyperbolic GCMS. Then, we start computing $K_2(A,F)$ for the $2\times 2$ hyperbolic GCMs (in Section 6), and then for the $3\times 3$ matrices (in Section 7) and so on (concluding with Section 8), until all the hyperbolic cases are covered.

\section{Preliminaries}

An indecomposable \emph{Generalised Cartan Matrix} (GCM) is called hyperbolic if it is not finite or affine, but any proper principal minor of the matrix is finite or affine (cf.\cite{IDLA}). In this paper, out goal is to compute $K_2(A,F)$ for any hyperbolic GCM $A$ and field $F$.

We start by recalling the appropriate presentations given by Matsumoto \cite{MAT} and Rehmann and Morita \cite{AMTTFKMG}.

\begin{theorem}
	
If $A$ is a finite GCM (i.e. a Cartan matrix) and $F$ is a field, then 

$$K_2(A,F) = \twopartdef { K_2 (F) } {A\neq C_n (n\geq 1)} {K_2(2,F)} {A=C_n (n\geq 1)}$$

where $C_n$ is used as in the standard notation for the finite-dimensional simple Lie algebras (with $C_1=A_1$ and $C_2=B_2$). Furthermore, these groups have the following presentations.

	$K_2(F)$ is the abelian group generated by the symbols $\{u,v\}$ for $u,v\in F^{*}$ with defining relations:
	
	(A1) $\{tu,v\}=\{t,v\}\{u,v\}$
	
	(A2) $\{t,uv\}=\{t,u\}\{t,v\}$
	
	(A3) $\{u,1-u\}=1$ if $u\neq 1$
	
	for all $t,u,v\in F^{*}$. We call $\{.,.\}$ satisfying these relations a \emph{Steinberg symbol}.
	
	Also, $K_2(2,F)$ is the abelian group generated by the symbols $\{u,v\}$ for $u,v\in F^{*}$ with defining relations:
	
	(B1) $\{t,u\}\{tu,v\}=\{t,uv\}\{u,v\}$
	
	(B2) $\{1,1\}=1$
	
	(B3) $\{u,v\}=\{u^{-1},v^{-1}\}$
	
	(B4) $\{u,v\}=\{u,(1-u)v\}$ if $u\neq 1$
	
	for all $t,u,v\in F^{*}$. We call $\{.,.\}$ satisfying these relations a \emph{Steinberg cocycle}.
\end{theorem}

\begin{theorem}
	For a GCM $A$ and a field $F$, $K_2(A,F)$ is the abelian group generated by $c_i(u,v)$ for $i=1,\ldots,n$ and $u,v\in F^{*}$ with defining relations:
	
	(L1) $c_i(t,u)c_i(tu,v)=c_i(t,uv)c_i(u,v)$
	
	(L2) $c_i(1,1)=1$ 
	
	(L3) $c_i(u,v)=c_i(u^{-1},v^{-1})$
	
	(L4) $c_i(u,v)=c_i(u,(1-u)v)$ (if $u\neq 1$)
	
	(L5) $c_i(u,v^{a_{ji}})=c_j(u^{a_{ij}},v)$
	
	(L6) $c_i(tu,v^{a_{ji}})=c_i(t,v^{a_{ji}})c_i(u,v^{a_{ji}})$
	
	(L7) $c_i(t^{a_{ji}},uv)=c_i(t^{a_{ji}},u)c_i(t^{a_{ji}},v)$ 
	
	for all $t,u,v\in F^{*}$ and $1\leq i\neq j\leq n$, and where $a_{ij}$ is the ij-th entry of the GCM $A$.
	
\end{theorem}

Note: We shall be using Carter's and Kac's notation for GCMs (cf. \cite{LAFAT} and \cite{IDLA}), where a fixed simple root corresponds to a column of the matrix. Some authors instead have simple roots corresponding to the rows of the matrix, in which case all the same results shall hold but will require the reader to transpose the appropriate definitions and theorems.

\section{An Alternative Way of Presenting $K_2(A,F)$}

In their paper Rehmann and Morita give, without proof, the following theorem.

\begin{theorem}\label{RMThm}
Let $A=(a_{ij})$ be an $n\times n$ GCM and $F$ a field. For each $i=1,\ldots,n$ define

$$L_i = \twopartdef { K_2 (F) } {a_{ki}\, \mbox{is odd for some}\,\, 1\leq k\leq n} {K_2(2,F)} {a_{ki}\, \mbox{is even for all}\,\, 1\leq k\leq n}$$

and define $J$ to be the subgroup of $L_1\times L_2\times\ldots\times L_n$ generated by $\{u,v^{a_{ji}}\}_i\{u^{a_{ij}},v\}_j^{-1}$ for all $u,v\in F^{*}$ and $1\leq i\neq j\leq n$, where $\{\cdot,\cdot\}_i$ is the Steinberg symbol/cocycle for $L_i$. Then

$$K_2(A,F)=\dfrac{L_1\times L_2\times\ldots\times L_n}{J}$$
\end{theorem}

This theorem will be the main result which we will use in our computations, so we shall prove it here.

First, we shall derive a few relations which hold in $K_2(2,F)$ (and hence $K_2(A,F)$) and which shall make life easier.

\begin{lemma}\label{rels}
Let $A=(a_{ij})$ be an $n\times n$ GCM. The following relations hold in $K_2(2,F)$, and hence in (each component of) $K_2(A,F)$, for all $t,u,v\in F^{*}$.

(i) $\{u,1\}=1$ and $\{1,u\}=1$

(ii) $\{t,u\}=\{u^{-1},t\}$

(iii) $\{t^2u,v\}=\{t^2,v\}\{u,v\}$ and $\{t,u^2v\}=\{t,u^2\}\{t,v\}$

(iv) $\{t^{2m},v\}=\{t^2,v\}^m$

(v) $\{tu,v^{2k}\}=\{t,v^{2k}\}\{u,v^{2k}\}$ and $\{t^{2k},uv\}=\{t^{2k},u\}\{t^{2k},v\}$ for $k\in\B{Z}$

(vi) $\{u^2,v\}=\{u,v\}\{v,u\}^{-1}=\{u,v^2\}$

(vii) $[u,v]_r:=\{u^r,v\}$ is a Steinberg cocycle for all $r\in\B{Z}$ and a Steinberg symbol for all even $r\in\B{Z}$.

\end{lemma}

\begin{proof}

Let $t,u,v\in F^{*}$.

(i) Using (B1), $\{u,1\}\{u,1\}=\{u,1\}\{1,1\}$ and since (B2) gives $\{1,1\}=1$ the result follows, and similarly $\{1,u\}=1$.

(ii) See Proposition 5.7(a) in \cite{MAT}.

(iii) For $\{t,u^2v\}=\{t,u^2\}\{t,v\}$, see Lemma 39 in \cite{SLOCG}. Then using this and (ii), we get 
$\{t^2u,v\}=\{v^{-1},t^2u\}=\{v^{-1},t^2\}\{v^{-1},u\}=\{t^2,v\}\{u,v\}$.

(vi) This follows easily from (iii), setting $u=t^{2m-2}$ and proceeding inductively.

(v) Using (B1) and (iii), we have $\{tu,v^{2k}\}=\{t,uv^{2k}\}\{u,v^{2k}\}\{t,u\}^{-1}=\{t,u\}\{t,v^{2k}\}\{u,v^{2k}\}\{t,u\}^{-1}=\{t,v^{2k}\}\{u,v^{2k}\}$.

Similarly, $\{t^{2k},uv\}=\{t^{2k},u\}\{t^{2k}u,v\}\{u,v\}^{-1}=\{t^{2k},u\}\{t^{2k},v\}\{u,v\}\{u,v\}^{-1}=\{t^{2k},u\}\{t^{2k},v\}$

(vi) See Proposition 5.7(a) in \cite{MAT}.

(vii) It is easy to see that every Steinberg symbol is a Steinberg cocycle (the details can be found in the proof of the next theorem), so when $r$ is even we just prove that we have a Steinberg symbol.

Case 1) $r$ even, $r=2s$

Then $[u,v]_r=\{u^r,v\}=\{u^2,v\}^s=[u,v]_2^s$, where the middle equality comes from (iv).

So it is sufficient to show that $[u,v]_2$ is a Steinberg symbol. Let $t,u,v\in F^{*}$.

(A1) $[tu,v]_2=\{t^2u^2,v\}=\{t^2,v\}\{u^2,v\}=[t,v]_2[u,v]_2$ by (iii).

(A2) $[t,uv]_2=\{t^2,uv\}=\{t^2,u\}\{t^2,v\}=[t,u]_2[t,v]_2$ by (v).

(A3) $[u,1-u]_2=\{u^2,1-u\}=\{u,(1-u)u\}\{u,1-u\}\{u,u\}^{-1}=\{u,u\}\{u,1-u\}\{u,u\}^{-1}=\{u,u\}\{u,1\}\{u,u\}^{-1}=1$ where we use B1), B4), B4) again, and then (i) to get the result, for $u\neq 1$.

So $[u,v]_r$ is a Steinberg symbol, and hence a Steinberg cocycle, for $r$ even.

Case 2) $r$ odd, $r=2s+1$

Then we have $[u,v]_r=\{u^{2s+1},v\}=\{u^{2s},v\}\{u,v\}=[u,v]_{2s}[u,v]_1$ using (iii),(iv).

Then $[u,v]_{2s}$ is a Steinberg cocycle by above and $[u,v]_1=\{u,v\}$ is a Steinberg cocycle by definition, which means that 
$[u,v]_r$ is a Steinberg cocycle for $r$ odd.

\end{proof}

With these relations in hand we can now prove Theorem \ref{RMThm}.

\begin{proof}

In this proof we will use Rehmann and Morita's presentation of $K_2(A,F)$.

We wish to define a homomorphism $\phi:L_1\times L_2\times\ldots L_n\rightarrow K_2(A,F)$ by $\phi(\{u,v\}_i)=c_i(u,v)$.

Since $K_2(A,F)$ is abelian, we have to show that

(1) If the ith column of $A$ contains an odd entry, then $c_i(u,v)$ is bimultiplicative and $c_i(u,1-u)=1$ if $u\neq 1$.

(2) If the ith column of $A$ contains all even entries, then $c_i(u,v)$ satisfies (B1)--(B4).

Claim (2) is immediate, as the relations (B1)--(B4) in $K_2(2,F)$ are completely analogous to (L1)--(L4) in $K_2(A,F)$.

For (1), using (L4) with $v=1$ gives $c_i(u,1-u)=c_i(u,1)=1$ by (i) of Lemma \ref{rels}. So we just need to show it is bimultiplicative.

By assumption there exists $j$ such that $a_{ji}$ is odd, so say $a_{ji}=2k+1$ for $k\in\B{Z}$. Then we have
\begin{equation} 
\begin{split}
c_i(tu,v) & = c_i(tu,v^{2k+1})c_i(tu,v^{2k})^{-1}\:\:\mbox{(By (iii))} \\
 & = c_i(t,v^{2k+1})c_i(u,v^{2k+1})c_i(t,v^{2k})^{-1}c_i(u,v^{2k})^{-1}\:\:\mbox{(By (L6) and (v))}\\
 & = c_i(t,v)c_i(t,v^{2k})c_i(u,v)c_i(u,v^{2k})c_i(t,v^{2k})^{-1}c_i(u,v^{2k})^{-1}\:\:\mbox{(By (iii))}\\
 & = c_i(t,v)c_i(u,v)
\end{split}
\nonumber
\end{equation}
Similarly, we can show that $c_i(t,uv)=c_i(t,u)c_i(t,v)$, concluding (1).

So $\phi$ is a well-defined homomorphism, and it is surjective since the $c_i(u,v)$ generate $K_2(A,F)$. 

The kernel of this homomorphism will come from the defining relations of $K_2(A,F)$, so by the First Isomorphism Theorem it is sufficient to show 
that (L1)--(L4) and (L6)--(L7) already hold in $L_1\times\ldots\times L_n$. Then, the kernel will simply be the preimages of (L5), which is exactly $J$.

If $L_i=K_2(2,F)$, then (L1)--(L4) already hold in $L_i$ since they are analogous to (B1)--(B4).

If $L_i=K_2(F)$ then (L1)--(L4) hold in $L_i$ since every Steinberg symbol is a Steinberg cocycle (c.f. \cite{MAT} Lemma 5.6).

If $L_i=K_2(F)$ then (L6) and (L7) come straight from bimultiplicativity.

Otherwise, $L_i=K_2(2,F)$ means than $a_{ij}$ is even for all $j\neq i$, which means we have (L6) and (L7) straight from (v) of the above Lemma.

So (L1)--(L4) and (L6)--(L7) all hold in $L_1\times\ldots\times L_n$, and hence the kernel of $\phi$ is $J$, giving

$$K_2(A,F)=\dfrac{L_1\times L_2\times\ldots\times L_n}{J}$$

\end{proof}

Note that (iii) and (vi) of Lemma \ref{rels} tell us that we can also write $J$ as being generated by the  $\{u^{a_{ji}},v\}_i\{u^{a_{ij}},v\}_j^{-1}$, and we shall use this various times throughout the remainder of this paper. In particular we can always assume $i<j$ amongst the generators of $J$.

\section{Some Easy Examples}

Using this new presentation, we can derive some immediate results.

\begin{proposition}
$K_2(A,F)=K_2(F)$ if $A$ is an $n\times n$ simply-laced indecomposable GCM (i.e. all off-diagonal entries in $A$ are either $0$ or $-1$) for $n\geq2$.
\end{proposition}
\begin{proof}

Firstly, we note that since every column of $A$ has a non-zero off-diagonal entry (as $A$ is indecomposable),  we have $L_i=K_2(F)$ for all $1\leq i\leq n$.

Furthermore, we have for each pair $1\leq i\neq j\leq n$ that each generator of $J$ is either $\{u,v^{-1}\}_i\{u^{-1},v\}^{-1}_j$ for all $u,v\in F^{*}$ or $\{u,v^{0}\}_i\{u^{0},v\}_j^{-1}$ for all $u,v\in F^{*}$. As we are in $K_2(F)$ for all $i$, the second of these is trivial, and the 
first is the same as $\{u,v\}_i^{-1}\{u,v\}_j$ for all $u,v\in F^{*}$. Hence, since $A$ is indecomposable, quotienting $L_1\times\ldots\times L_n=K_2(F)\times\ldots\times K_2(F)$ by $J$ is the same 
as quotienting out by the equivalence relation $\{u,v\}_i\sim\{u,v\}_j$ for all $1\leq i,j\leq n$ and $u,v\in F^{*}$. Since the $\{u,v\}_i$ generate $L_i$, this clearly gives us that 

$$K_2(A,F)=\dfrac{K_2(F)\times\ldots\times K_2(F)}{J}=K_2(F)$$
\end{proof}

\begin{proposition}\label{del}
If $A=(a_{ij})$ is a $n\times n$ GCM with $a_{s1}=-1$ for some $2\leq s \leq n$, then we have

$$K_2(A,F)=\dfrac{L_2\times\ldots\times L_n}{J'}$$

where the $L_i$ are defined as usual and $J'$ is generated by $\{u,v^{a_{ji}}\}_i\{u^{a_{ij}},v\}_j^{-1}$ for all $u,v\in F^{*}$ and $2\leq i<j\leq n$, and $\{u^{a_{j1}a_{1i}},v\}_i\{u^{a_{i1}a_{1j}},v\}_j^{-1}$ for all $u,v\in F^{*}$ and $2\leq i<j\leq n$ .
\end{proposition}
\begin{proof}

If $a_{s1}=-1$ then $L_1=K_2(F)$, and we can write $J$ as being generated by $\{u,v^{a_{ji}}\}_i\{u^{a_{ij}},v\}_j^{-1}$ for all $u,v\in F^{*}$ and $2\leq i<j\leq n$, and $\{u,v\}_1^{a_{j1}}\{u^{a_{j1}},v\}_j^{-1}$ for all $u,v\in F^{*}$ and $2\leq j\leq n$.

We can now define a map $$f:\dfrac{L_1\times L_2\times\ldots\times L_n}{J}\rightarrow\dfrac{L_2\times\ldots\times L_n}{J'}$$ by
$f(\{u,v\}_i=\{u,v\}_i$ for $i\neq 1$ and $f(\{u,v\}_1)=\{u,v^{-a_{1s}}\}_s$ for all $u,v\in F^{*}$.

This is a well-defined homomorphism because it is easy to see that the images of the generators satisfy the defining relations of the domain, and it is clearly surjective (details are left to the reader).

We can also define a map $$g:\dfrac{L_2\times\ldots\times L_n}{J'}\rightarrow\dfrac{L_1\times L_2\times\ldots\times L_n}{J}$$ by 
$g(\{u,v\}_i)=\{u,v\}_i$ for all $i\neq t$ and $u,v\in F^{*}$. 

This is well-defined homomorphism in a similar way as above and is clearly a left inverse to $f$ (again, the details are left to the reader).

Hence $f$ is an isomorphism and the result follows from Theorem \ref{RMThm}.
\end{proof}

Note that by simultaneously reordering the rows and columns of the matrix if necessary, this same argument holds for a matrix with a $-1$ entry in any column.

In particular, in the case when $A$ has a column (say the t-th) in which the only non-zero off-diagonal entry is $a_{st}=-1$, the proposition will give us that $K_2(A,F)=K_2(A',F)$, where 
$A'$ is the matrix obtained from $A$ by deleting the t-th row and column, except in the case where $a_{ts}$ is odd and all other entries in this column are even. In this case, deleting the t-th row and column would change the $L_t$ from being $K_2(F)$ to $K_2(2,F)$, which means that we wouldn't get $K_2(A',F)$.

Using this process, it is easy to verify Matsumoto's Theorem and obtain simpler presentations for most of the affine GCMs.

\begin{example}
If $A$ is a finite GCM, then

$$K_2(A,F) = \twopartdef { K_2 (F) } {A\neq C_n (n\geq 1)} {K_2(2,F)} {A=C_n (n\geq 1)}$$

where $C_n$ is used as in the standard notation for the finite-dimensional simple Lie algebras (with $C_1=A_1$ and $C_2=B_2$).

If $A$ is an affine GCM, then

$$K_2(A,F) = \twopartdef { K_2 (F) } {A\neq \tilde{C_n'},\tilde{C_n} (n\geq 1)} {K_2(2,F)} {A=\tilde{C_n'} (n\geq 1)}$$

where we use Carter's notation for the affine Kac-Moody Algebras (cf.\cite{LAFAT}) and setting $\tilde{C_1'}=\tilde{A_1'}$ and $\tilde{C_1}=\tilde{A_1}$ 

Both these results can be seen just by iterations of the above result, each time using a column with a -1 as the single non-zero off-diagonal entry. 

By the same process, we can also show that $K_2(\tilde{C_n},F)=K_2(\tilde{A_1},F)$ for all $n\geq 1$, although computing what that actually is requires a little more work. We shall see how to do this in Chapter 6.
\end{example}

\section{Odd Columns}

We now aim to obtain a simplification of Rehmann and Morita's presentation in the case when all the columns of $A=(a_{ij})$ (an indecomposable $n\times n$ GCM) have an odd entry. Recall that we are using the definition of Kac-Moody algebras from Carter \cite{LAFAT} where $a_{ij}=\alpha_j(h_i)$, where $\alpha_j$ is a simple root and $h_i$ is a simple coroot. Some authors instead use $a_{ij}=\alpha_i(h_j)$, in which case the same theorem applies but requires an odd entry in every row of the matrix rather than every column.

Theorem \ref{RMThm} then tells us that

$$K_2(A,F)=\dfrac{K_2(F)\times\ldots\times K_2(F)}{J}$$

where $J$ is defined to be the subgroup of $K_2(F)\times\ldots\times K_2(F)$ generated by $\{u,v\}_i^{a_{ji}}\{u,v\}_j^{-a_{ij}}$ for all $u,v\in F^{*}$ and $1\leq i<j\leq n$, with $\{\cdot,\cdot\}_i$ the Steinberg symbol for the ith summand.

\begin{theorem}\label{odd}

Let $A=(a_{ij})$ be an indecomposable GCM such that every column contains an odd entry, and let $F$ be a field. Let G be the abelian group defined by $G=\langle x_1,\ldots,x_n\,\vert\, x_i^{a_{ji}}=x_j^{a_{ij}}, [x_i,x_j]=1\; \mbox{for}\; 1\leq i<j\leq n\rangle$.

Suppose $G\cong \B{Z}/{r_1\B{Z}}\times\ldots\times\B{Z}/{r_s\B{Z}}\times\B{Z}^{n-s}$ for some positive integers $r_1,\ldots,r_s$ and $s\in\{0,1,\ldots,n\}$.

Then $K_2(A,F)\cong \dfrac{K_2(F)}{r_1K_2(F)}\times\ldots\times\dfrac{K_2(F)}{r_sK_2(F)}\times K_2(F)^{n-s}$.
\end{theorem}

\begin{proof}

Since $K_2(A,F)=\dfrac{K_2(F)\times\ldots\times K_2(F)}{J}$ we can present $K_2(A,F)$ as being the abelian group generated by symbols $\{u,v\}_i$ for $u,v\in F^{*}$ and $1\leq i\leq n$ with defining relations

(X1) $\{tu,v\}_i=\{t,v\}_i\{u,v\}_i$

(X2) $\{t,uv\}_i=\{t,u\}_i\{t,v\}_i$

(X3) $\{u,1-u\}_i=1$ for $u\neq 1$

(X4) $\{u,v\}_i^{a_{ji}}=\{u,v\}_j^{a_{ij}}$

for all $t,u,v\in F^{*}$ and $1\leq i< j\leq n$.

We can also present $\dfrac{K_2(F)}{r_1K_2(F)}\times\ldots\times\dfrac{K_2(F)}{r_sK_2(F)}\times K_2(F)^{n-s}$ as being the abelian group generated by symbols $[u,v]_i$ for $u,v\in F^{*}$ and $1\leq i\leq n$ with defining relations:

(Y1) $[tu,v]_i=[t,v]_i[u,v]_i$

(Y2) $[t,uv]_i=[t,u]_i[t,v]_i$

(Y3) $[u,1-u]_i=1$ for $u\neq 1$

for all $t,u,v\in F^{*}$ and $1\leq i\leq n$, and 

Y4) $[u,v]_i^{r_i}=1$ 

for all $u,v\in F^{*}$ and $1\leq i\leq s$.

Let $\phi:G\rightarrow \B{Z}/{r_1\B{Z}}\times\ldots\times\B{Z}/{r_s\B{Z}}\times\B{Z}^{n-s}$ be the isomorphism from the statement of the theorem, with inverse $\psi:\B{Z}/{r_1\B{Z}}\times\ldots\times\B{Z}/{r_s\B{Z}}\times\B{Z}^{n-s}\rightarrow G$.

Suppose $\phi(x_i)=y_1^{\mu_{1i}}\ldots y_n^{\mu_{ni}}$ for $1\leq i\leq n$, where the $y_i$ are the generators of the cyclic summands of $\B{Z}/{r_1\B{Z}}\times\ldots\times\B{Z}/{r_s\B{Z}}\times\B{Z}^{n-s}$, and that
$\psi(y_i)=x_1^{\nu_{1i}}\ldots x_n^{\nu_{ni}}$ for $1\leq i\leq n$, where $\mu_{ji},\nu_{ji}\in\B{Z}$.

Define $\Phi:K_2(A,F)\rightarrow \dfrac{K_2(F)}{r_1K_2(F)}\times\ldots\times\dfrac{K_2(F)}{r_sK_2(F)}\times K_2(F)^{n-s}$ by $$\Phi(\{u,v\}_i)=[u,v]_1^{\mu_{1i}}\ldots [u,v]_n^{\mu_{ni}}$$ for each $u,v\in F^{*}$ and $1\leq i\leq n$.

To show that this is a well defined homomorphism, we need to show that\\ $[u,v]_1^{\mu_{1i}}\ldots [u,v]_n^{\mu_{ni}}$ satisfies (X1)--(X4).

Since everything is abelian and (X1)--(X3) are the same as (Y1)--(Y3), we immediately see that $[u,v]_1^{\mu_{1i}}\ldots [u,v]_n^{\mu_{ni}}$ satisfies (X1)--(X3).

For (X4) we need that for each $u,v\in F^{*}$ and $1\leq i< j\leq n$ we have $$[u,v]_1^{\mu_{1i}a_{ji}}\ldots [u,v]_n^{\mu_{ni}a_{ji}}=[u,v]_1^{\mu_{1j}a_{ij}}\ldots [u,v]_n^{\mu_{nj}a_{ij}}$$

But we know, since $\phi$ is well-defined and $x_i^{a_{ji}}=x_j^{a_{ij}}$, that $y_1^{\mu_{1i}a_{ji}}\ldots y_n^{\mu_{ni}a_{ji}}=y_1^{\mu_{1j}a_{ij}}\ldots y_n^{\mu_{nj}a_{ij}}$.

Since the defining relations satisfied by the $y_i$ are also satisfied by the $[u,v]_i$, we get that the above equality must hold in $\dfrac{K_2(F)}{r_1K_2(F)}\times\ldots\times\dfrac{K_2(F)}{r_sK_2(F)}\times K_2(F)^{n-s}$, as required.

Now, define $\Psi:\dfrac{K_2(F)}{r_1K_2(F)}\times\ldots\times\dfrac{K_2(F)}{r_sK_2(F)}\times K_2(F)^{n-s}\rightarrow K_2(A,F)$ by $$\Psi([u,v]_i)=\{u,v\}_1^{\nu_{1i}}\ldots \{u,v\}_n^{\nu_{ni}}$$ for each $u,v\in F^{*}$ and $1\leq i\leq n$.

To show that this is a well defined homomorphism, we need to show that $\{u,v\}_1^{\nu_{1i}}\ldots \{u,v\}_n^{\nu_{ni}}$ satisfies (Y1)--(Y4).

Since everything is abelian and (Y1)--(Y3) are the same as (X1)--(X3), we immediately see that $\{u,v\}_1^{\nu_{1i}}\ldots \{u,v\}_n^{\nu_{ni}}$ satisfies (Y1)--(Y3).

For (Y4) we need that for each $u,v\in F^{*}$ and $1\leq i\leq s$ we have $$\{u,v\}_1^{\nu_{1i}r_i}\ldots \{u,v\}_n^{\nu_{ni}r_i}=1$$

But we know, since $\psi$ is well-defined and $y_i^{r_i}=1$ for $1\leq i\leq s$, that $x_1^{\nu_{1i}r_i}\ldots x_n^{\nu_{ni}r_i}=1$.

Since the defining relations satisfied by the $x_i$ are also satisfied by the $\{u,v\}_i$, we get that the above equality must hold in $K_2(A,F)$, as required.

So we have two well-defined homomorphisms $\Phi$ and $\Psi$.

Now, we have for any $u,v\in F^{*}$ and $1\leq i\leq n$ that $$\Psi(\Phi(\{u,v\}_i))=\Psi([u,v]_1^{\mu_{1i}}\ldots [u,v]_n^{\mu_{ni}})=\{u,v\}_1^{\mu_{1i}\nu_{11}+\ldots+\mu_{ni}\nu_{1n}}\ldots\{u,v\}_n^{\mu_{ni}\nu_{n1}+\ldots+\mu_{ni}\nu_{nn}}$$

But since $\psi(\phi(x_i))=x_i$ for any $1\leq i\leq n$, we have $x_1^{\mu_{1i}\nu_{11}+\ldots+\mu_{ni}\nu_{1n}}\ldots x_n^{\mu_{ni}\nu_{n1}+\ldots+\mu_{ni}\nu_{nn}}=x_i$ for any $1\leq i\leq n$, and hence, since the defining relations for the 
$x_i$ are satisfied by the $\{u,v\}_i$, we get that $$\Psi(\Phi(\{u,v\}_i))=\{u,v\}_i$$ for all $u,v\in F^{*}$ and $1\leq i\leq n$.

By a similar argument, we can get that $\Phi(\Psi([u,v]_i))=[u,v]_i$ for all $u,v\in F^{*}$ and $1\leq i\leq n$.

Hence, as the $\{u,v\}_i$ generate $K_2(A,F)$ and the $[u,v]_i$ generate $\dfrac{K_2(F)}{r_1K_2(F)}\times\ldots\times\dfrac{K_2(F)}{r_sK_2(F)}\times K_2(F)^{n-s}$, we can conclude that $\Phi$ and $\Psi$ are inverse
homomorphisms.

Hence $\Phi$ is an isomorphism and so $$K_2(A,F)\cong \dfrac{K_2(F)}{r_1K_2(F)}\times\ldots\times\dfrac{K_2(F)}{r_sK_2(F)}\times K_2(F)^{n-s}$$
\end{proof}

The above proof could be simplified by observing that as we can obtain the elements $y_i$ from the elements $x_j$ in $G$ by performing unimodular row and column operations to an appropriate matrix (and vice versa), we can similarly obtain (for fixed $u,v\in F^{*}$) the $[u,v]_i$ from the $\{u,v\}_j$ by performing those same operations (and vice versa). The fact that these elements are obtained in this way would cut down the proof that the given maps are well-defined and inverses - since we know, for example, that $\Psi$ just inverts the operations of $\Phi$.

Nonetheless, we have given the more explicit version of this proof since for later proofs (where the analogy with the Fundamental Theorem of Finitely Generated Abelian Groups is not so strong and cannot be applied so easily) we will have a greater need to understand the details of why given maps are isomorphisms.

\
\begin{example}
	Consider the following hyperbolic GCM
	
	$$A=\begin{pmatrix}
	2 & -1 & -3\\
	-3 & 2 & -1\\
	-1 & -3 & 2\\
	\end{pmatrix}$$ 
	
	This has an odd entry in every column, so we can apply the above theorem. We form the abelian group

	$$G:=\langle x,y,z\,\vert\,x^{-3}=y^{-1},\,x^{-1}=z^{-3},\,y^{-3}=z^{-1}, xy=yx, xz=zx, yz=zy\rangle$$
	
	It can easily be checked, for example using Sage or Magma, that
	$$G\cong \B{Z}/2\B{Z}\times\B{Z}/13\B{Z}$$
	
	Hence the above theorem tells us that
	$$K_2(A,F)\cong \dfrac{K_2(F)}{2K_2(F)}\times \dfrac{K_2(F)}{13K_2(F)}$$
	
	In particular, if $F$ is a local field, then we get
	$$K_2(A,F)\cong \mu_2\times\mu_{13}$$
	
	where $\mu_m=\{u\in F^{*}\,\vert\,u^m=1\}$ (cf. Section 5, Example 7 in \cite{AMTTFKMG}).
	
\end{example}.

\section{Rank 2 GCMs}

Since there a unique $1\times 1$ GCM, which is finite, every $2\times 2$ GCM is either finite, affine or hyperbolic. Since our goal in this paper is to 
give presentations of $K_2(A,F)$ for all the hyperbolic GCMs $A$, we aim in this chapter to give a characterisation of $K_2(A,F)$ for all $2\times 2$ GCMs.

For this chapter, we will be interested in matrices of the form

$$A=\begin{pmatrix}
2 & -b\\
-a & 2\\
\end{pmatrix}$$

where $a,b\in\B{Z}_{>0}$.

If $a$ and $b$ are both odd then each column has an odd entry, which means that we can apply Theorem \ref{odd} to get

$$K_2(A,F)\cong \dfrac{K_2(F)}{hK_2(F)}\times K_2(F)$$

where $h=hcf(a,b)$ (since $\langle x,y\,\vert\,x^a=y^b\rangle\cong \B{Z}/h\B{Z}\times \B{Z}$ as abelian groups).

\begin{theorem}

In the case where $A=\begin{pmatrix}
2 & -b\\
-a & 2\\
\end{pmatrix}$ with both $a$ and $b$ even, we have

$$K_2(A,F)\cong K_2(2,F)\times\dfrac{K_2(2,F)}{h\langle\{u^2,v\}\,\vert\, u,v\in F^{*}\rangle}$$

where $h=hcf(\frac{a}{2},\frac{b}{2})$. 

\end{theorem}

\begin{proof}
	Using Theorem \ref{RMThm} and Lemma \ref{rels}(vi) and setting $c=\frac{a}{2}$ and $d=\frac{b}{2}$, we can write $$K_2(A,F)=\dfrac{K_2(2,F)\times K_2(2,F)}{J}$$ where $J$ is the (normal) subgroup generated by $\{u^2,v\}_1^{c}\{u^2,v\}_2^{-d}$ for all $u,v\in F^{*}$.

Hence we can present $K_2(A,F)$ as being the abelian group generated by symbols $\{u,v\}_i$ for $u,v\in F^{*}$ and $i=1,2$ with defining relations:

(R1) $\{t,u\}_i\{tu,v\}_i=\{t,uv\}_i\{u,v\}_i$

(R2) $\{1,1\}_i=1$

(R3) $\{u,v\}_i=\{u^{-1},v^{-1}\}_i$

(R4) $\{u,v\}_i=\{u,(1-u)v\}_i$ if $u\neq 1$

(R5) $\{u^2,v\}_1^c=\{u^2,v\}_2^d$

for all $t,u,v\in F^{*}$ and $i=1,2$ (using Lemma \ref{rels} (iv) for (R5)).

Also, we can present $K_2(2,F)\times\dfrac{K_2(2,F)}{h\langle\{u^2,v\}\,\vert\, u,v\in F^{*}\rangle}$ as being the abelian group generated by the symbols $[u,v]_i$ for $u,v\in F^{*}$ and $i=1,2$ subject to the defining relations:

(S1) $[t,u]_i[tu,v]_i=[t,uv]_i[u,v]_i$

(S2) $[1,1]_i=1$

(S3) $[u,v]_i=[u^{-1},v^{-1}]_i$

(S4) $[u,v]_i=[u,(1-u)v]_i$ if $u\neq 1$

(S5) $[u^2,v]_2^h=1$

for all $t,u,v\in F^{*}$ and $i=1,2$ (again, using Lemma \ref{rels} (iv)).

Suppose that $rc+sd=h$ for $r,s\in\B{Z}$. Define the map $\Phi:K_2(A,F)\rightarrow K_2(2,F)\times\dfrac{K_2(2,F)}{h\langle\{u^2,v\}\,\vert\, u,v\in F^{*}\rangle}$ by $$\Phi(\{u,v\}_1)=[u,v]_1^{\frac{d}{h}}[u,v]_2^r$$ $$\Phi(\{u,v\}_2)=[u,v]_1^{\frac{c}{h}}[u,v]_2^{-s}$$ for all $u,v\in F^{*}$.

To show that this is a well defined homomorphism, we need to show that (R1)--(R4) are satisfied by $[u,v]_1^{\frac{d}{h}}[u,v]_2^r$ and $[u,v]_1^{\frac{c}{h}}[u,v]_2^{-s}$, which is clear from (S1)--(S4), and that 
$[u^2,v]_1^{\frac{cd}{h}}[u^2,v]_2^{rc}=[u^2,v]_1^{\frac{cd}{h}}[u^2,v]_2^{-sd}$, which follows immediately from $rc+sd=h$ and (S5).

Now define the map $\Psi:K_2(2,F)\times\dfrac{K_2(2,F)}{h\langle\{u^2,v\}\,\vert\, u,v\in F^{*}\rangle}\rightarrow K_2(2,F)$ by $$\Psi([u,v]_1)=\{u,v\}_1^s\{u,v\}_2^r$$ $$\Psi([u,v]_2)=\{u,v\}_1^{\frac{c}{h}}\{u,v\}_2^{\frac{-d}{h}}$$ for all $u,v\in F^{*}$.

To show that this is a well defined homomorphism, we need to show that (S1)--S4) are satisfied by $\{u,v\}_1^s\{u,v\}_2^r$ and $\{u,v\}_1^{\frac{c}{h}}\{u,v\}_2^{\frac{-d}{h}}$, which is clear from (R1)--(R4), and that 
$\{u^2,v\}_1^c\{u^2,v\}_2^{-d}=1$, which follows immediately from (R5).

So we have two well-defined homomorphisms, $\Phi$ and $\Psi$.

We now see that $$\Phi(\Psi([u,v]_1)=\Phi(\{u,v\}_1^s\{u,v\}_2^r)=[u,v]_1^{\frac{ds}{h}}[u,v]_2^{rs}[u,v]_1^{\frac{cr}{h}}[u,v]_2^{-rs}$$
$$=[u,v]_1^{\frac{cr+ds}{h}}=[u,v]_1$$

It can be similarly shown that $\Phi(\Psi([u,v]_2)=[u,v]_2$, $\Psi(\Phi(\{u,v\}_1))=\{u,v\}_1$ and $\Psi(\Phi(\{u,v\}_2))=\{u,v\}_2$.

Hence, as the $\{u,v\}_i$ generate $K_2(2,F)$ and the $[u,v]_i$ generate $K_2(2,F)\times\dfrac{K_2(2,F)}{h\langle\{u^2,v\}\,\vert\, u,v\in F^{*}\rangle}$, we get that $\Phi$ and $\Psi$ are inverse homomorphisms. 

So $K_2(A,F)\cong K_2(2,F)\times\dfrac{K_2(2,F)}{h\langle\{u^2,v\}\,\vert\, u,v\in F^{*}\rangle}$.
\end{proof}

\begin{example}
	This finally allows us to compute $K_2(\tilde{A_1},F)$. The matrix of $\tilde{A_1}$ is
	
	$$\tilde{A_1}=\begin{pmatrix}
		2 & -2\\
		-2 & 2\\
	\end{pmatrix}$$
	
	So since $hcf(1,1)=1$, the above gives us that
	
	$$K_2(\tilde{A_1},F)\cong K_2(2,F)\times\dfrac{K_2(2,F)}{\langle\{u^2,v\}\,\vert\, u,v\in F^{*}\rangle}$$
	
	It has been shown by Suslin \cite{TIKF} that in fact 
	
	$$\dfrac{K_2(2,F)}{\langle\{u^2,v\}\,\vert\, u,v\in F^{*}\rangle}\cong I^2(F)$$
	
	where $I^2(F)$ is the fundamental ideal of the Witt ring $W(F)$ of $F$, so we could also write	
	$$K_2(\tilde{A_1},F)\cong K_2(2,F)\times I^2(F)$$
	
\end{example}

\begin{theorem}
	
	In the case where $A=\begin{pmatrix}
	2 & -b\\
	-a & 2\\
	\end{pmatrix}$ with $a$ odd and $b$ even (or equivalently, $a$ even and $b$ odd), we have
	
	$$K_2(A,F)\cong \dfrac{K_2(F)}{hK_2(F)} \times K_2(2,F)$$
	
	where $h=hcf(a,b)$.
	
\end{theorem}

\begin{proof}
	Using Theorem \ref{RMThm} and Lemma \ref{rels}(vi) we can write $$K_2(A,F)=\dfrac{K_2(F)\times K_2(2,F)}{J}$$ where $J$ is the (normal) subgroup generated by $\{u^a,v\}_1\{u^{-b},v\}_2$ for all $u,v\in F^{*}$.
	
	In particular, we can present $K_2(A,F)$ as being the abelian group generated by $\{u,v\}_i$ for $u,v\in F^{*}$ and $i=1,2$ with defining relations:
	
	(R1) $\{.,.\}_1$ is a Steinberg symbol
	
	(R2) $\{.,.\}_2$ is a Steinberg cocycle
	
	(R3) $\{u^a,v\}_1=\{u^b,v\}_2$ for all $u,v\in F^{*}$
	
	We can also present  $\dfrac{K_2(F)}{hK_2(F)}\times K_2(2,F)$ as being the abelian group generated by $[u,v]_i$ for $u,v\in F^{*}$ and $i=1,2$ with defining relations:
	
	(S1) $[.,.]_1$ is a Steinberg symbol
	
	(S2) $[.,.]_2$ is a Steinberg cocycle
	
	(S3) $[u^h,v]_1=1$ for all $u,v\in F^{*}$
	
	Suppose that $h=ra+sb$ for $r,s\in \B{Z}$.
	
	Define a map $\Phi:K_2(A,F)\rightarrow \dfrac{K_2(F)}{hK_2(F)}\times K_2(2,F)$ by
	
	$$\Phi(\{u,v\}_1)=[u^r,v]_1[u^{\frac{b}{h}},v]_2$$	
	$$\Phi(\{u,v\}_2)=[u^{-s},v]_1[u^{\frac{a}{h}},v]_2$$
	
	Since $\frac{b}{h}$ is even, Lemma \ref{rels}(vii) tells us that $[u^r,v]_1[u^{\frac{b}{h}},v]_2$ is a Steinberg symbol, and
	Lemma \ref{rels}(vii) also tells is that $[u^{-s},v]_1[u^{\frac{a}{h}},v]_2$ is a Steinberg cocycle.
	
	Furthermore, we have $$\Phi(\{u^a,v\}_1)=[u^{ar},v]_1[u^{\frac{ab}{h}},v]_2=[u^{h},v]_1[u^{-bs},v]_1[u^{\frac{ab}{h}},v]_2=\Phi(\{u^b,v\}_2)$$
	
	Hence $\Phi$ is a well defined homomorphism.
	
	Now, define a map $\Psi:\dfrac{K_2(F)}{hK_2(F)}\times K_2(2,F)\rightarrow K_2(A,F)$ by
	
	$$\Psi([u,v]_1)=\{u^{\frac{a}{h}},v\}_1\{u^{-\frac{b}{h}},v\}_2$$
	$$\Psi([u,v]_2)=\{u^s,v\}_1\{u^r,v\}_2$$
	
	Since $\frac{b}{h}$ is even, Lemma \ref{rels}(vii) tells us that $\{u^{\frac{a}{h}},v\}_1\{u^{-\frac{b}{h}},v\}_2$ is a Steinberg symbol, and
	Lemma \ref{rels}(vii) also tells is that $\{u^s,v\}_1\{u^r,v\}_2$ is a Steinberg cocycle.
	
	Furthermore, we have 	
	$$\Psi([u^h,v]_1)=\{u^a,v\}_1\{u^{-b},v\}_2=1$$
	
	Hence $\Psi$ is a well-defined homomorphism.
	
	Now we note that	
	\begin{equation} 
	\begin{split}
	\Psi(\Phi(\{u,v\}_1)) & = \Psi ([u^r,v]_1 [u^{\frac{b}{h}},v]_2) \\
	& = \{u^{\frac{ra}{h}},v\}_1 \{u^{-\frac{br}{h}},v\}_2 \{u^{\frac{bs}{h}},v\}_1 \{u^{\frac{br}{h}},v\}_2\\
	& = \{u^{\frac{ra+bs}{h}},v\} \:\:\mbox{(By (A1) and (iii))}\\
	& = \{u,v\}_1
	\end{split}
	\end{equation}

	By a similar process, we get that $\Psi(\Phi(\{u,v\}_2))=\{u,v\}_2$, $\Phi(\Psi([u,v]_1))=[u,v]_1$ and $\Phi(\Psi([u,v]_2))=[u,v]_2$.
	
	Hence $\Phi$ and $\Psi$ are inverse homomorphisms and we get
	
	$$K_2(A,F)\cong \dfrac{K_2(F)}{hK_2(F)} \times K_2(2,F)$$
	
	where $h=hcf(a,b)$. 	
	
\end{proof}

\section{Rank 3 Hyperbolic GCMs}

From now on we shall be referring to Chung, Carbone, et.al.'s paper \emph{Classification of Hyperbolic Dynkin Diagrams, Root Lengths 
and Weyl Group Orbits} in our quest to find more concise presentations for $K_2(A,F)$ for the hyperbolic GCMs $A$. In particular, we shall be using the 
numbering of these GCMs (equivalently Dynkin diagrams) as given in this paper, although since $K_2(A,F)$ is unchanged by simultaneously reordering rows and columns of $A$ we may at 
times use a different labelling of vertices of the Dynkin diagrams than those given in the above paper. 

In this paper, the authors list all 238 hyperbolic Dynkin diagrams of rank 3 or higher. Of these, 123 have rank 3.

We have already made some steps towards determining the $K_2(A,F)$ for these $A$; namely in the cases when all the columns of $A$ contain an odd element and most of the cases when $A$ 
contains a column with a -1 as its single non-zero off-diagonal entry. A quick survey of the table in \cite{CHDD} tells us that 67 of the $3\times 3$ GCMs fall into the first of these categories, with another 12 covered by the second category. So we have
 46 remaining cases to cover. We start by considering some classes of $3\times 3$ GCMs.

Class 1:

$A=\begin{pmatrix}
2 & -a & -b\\
-c & 2 & -d\\
-1 & -1 & 2\\
\end{pmatrix}$

where $a,b,c,d\in\B{Z}_{>0}$.

Since we have already covered the case where all columns have an odd element, we only need to treat the case where $b$ and $d$ are even. 

Then Theorem \ref{RMThm} gives us that

$K_2(A,F)=\dfrac{K_2(F)\times K_2(F)\times K_2(2,F)}{J}$

where in this case $J$ corresponds to the following equivalence relation:

For all $u,v\in F^{*}$, we have (taking inverses to remove any negative powers, and using Lemma \ref{rels} (iii) to take indices outside of the symbol)

\begin{itemize}
\item $\{u,v\}_1^c=\{u,v\}_2^a$
\item $\{u,v\}_1=\{u^2,v\}_3^{\frac{b}{2}}$
\item $\{u,v\}_2=\{u^2,v\}_3^{\frac{d}{2}}$
\end{itemize}

From the second and third of these three relations it is clear that we can write $K_2(A,F)$ as a quotient of $K_2(2,F)$ (as the $\{u,v\}_1$ and $\{u,v\}_2$ generate the $K_2(F)$).

In particular, we will have $K_2(A,F)$ being $K_2(2,F)$ quotiented out by the relation

\begin{itemize}
\item $\{u^2,v\}_3^{\frac{bc}{2}}=\{u^2,v\}_3^{\frac{ad}{2}}$
\end{itemize}

for all $u,v\in F^{*}$.

Or equivalently,

$$K_2(A,F)=\dfrac{K_2(2,F)}{\left(\frac{ad-bc}{2}\right)\langle \{u^2,v\}\,\vert u,v\in F^{*}\rangle}$$

This covers another 18 cases, leaving 28 remaining.

Class 2:

$A=\begin{pmatrix}
2 & -1 & -a\\
-b & 2 & -c\\
-1 & -d & 2\\
\end{pmatrix}$

where $a,b,c,d\in\B{Z}_{>0}$.

Again, we only need to treat the case when $a$ and $c$ are both even.

Here, we can apply the idea of Proposition \ref{del} twice in order to get

$$K_2(2,F)=\dfrac{K_2(F)\times K_2(F)\times K_2(2,F)}{J}=\dfrac{K_2(F)\times K_2(2,F)}{J'}=\dfrac{K_2(2,F)}{J''}$$

where $J'$ is generated by $\{u,v\}_2^d\{u^2,v\}_3^{\frac{-c}{2}}$ and $\{u,v\}_2\{u^2,v\}_3^{-\frac{ab}{2}}$ for all $u,v\in F^{*}$, 
and $J''$ is generated by $\{u^2,v\}_3^{\frac{c}{2}}\{u^2,v\}_3^{-\frac{abd}{2}}$ for all $u,v\in F^{*}$, which hence gives us

$$K_2(A,F)=\dfrac{K_2(2,F)}{\left(\frac{abd-c}{2}\right)\langle \{u^2,v\}\,\vert u,v\in F^{*}\rangle}$$

This covers another 8 cases, leaving 20 more.

Class 3:

$A=\begin{pmatrix}
2 & -2 & -2\\
-2 & 2 & -2\\
-2 & -2 & 2\\
\end{pmatrix}$    and   $A=\begin{pmatrix}
2 & -2 & 0\\
-2 & 2 & -2\\
0 & -2 & 2\\
\end{pmatrix}$

It is easy to see that $K_2(A,F)$ will be the same for these two matrices. In this case, we get that

$$K_2(A,F)\cong K_2(2,F)\times \dfrac{K_2(2,F)}{\langle \{u^2,v\} \,\vert \, u,v\in F^{*}\rangle} \times \dfrac{K_2(2,F)}{\langle \{u^2,v\}\,\vert u,v\in F^{*}\rangle}$$
	
The proof of this is very similar to the proofs in the rank 2 cases. Namely, we can easily construct explicit inverse homomorphisms between the two groups by defining them on generators and checking that they are well defined and inverses. The details of this proof are left to the reader.

For the remaining 18 matrices, we shall simply give a table containing the final results. These were all obtained from simple manipulations of the relations in the groups in order to give slightly different presentations, similar to Classes 1 and 2 above, and then constructing explicit isomorphisms.

The numbering in the table follows the numbering in \cite{CHDD}.

\begin{adjustwidth}{-4em}{}
\begin{center}
	\begin{tabular}{ | m{1.5cm} | m{10em}| m{25em} | } 
		\hline
		Number & Matrix $A$ & $K_2(A,F)$ \\ 
		\hline
		27. & $A=\begin{pmatrix}
			2 & -1 & 0\\
			-3 & 2 & -2\\
			0 & -1 & 2\\
		\end{pmatrix}$ & $K_2(2,F)$ \\ 
		\hline
		40. & $A=\begin{pmatrix}
		2 & -1 & -2\\
		-1 & 2 & -2\\
		-2 & -2 & 2\\
		\end{pmatrix}$ & $K_2(F)\times \dfrac{K_2(2,F)}{\langle \{u^2,v\} \,\vert \, u,v\in F^{*}\rangle}$ \\
		\hline
		55. & $A=\begin{pmatrix}
		2 & -1 & -2\\
		-2 & 2 & -2\\
		-2 & -1 & 2\\
		\end{pmatrix}$ & $K_2(2,F)\times \dfrac{K_2(2,F)}{\langle \{u^2,v\} \,\vert \, u,v\in F^{*}\rangle}$ \\
		\hline
		59. & $A=\begin{pmatrix}
		2 & -1 & -2\\
		-2 & 2 & -2\\
		-2 & -2 & 2\\
		\end{pmatrix}$ & $\dfrac{K_2(2,F)}{\langle \{u^2,v\} \,\vert \, u,v\in F^{*}\rangle}\times \dfrac{K_2(2,F)}{\langle \{u^2,v\} \,\vert \, u,v\in F^{*}\rangle}$ \\
		\hline
		63. & $A=\begin{pmatrix}
		2 & -2 & -3\\
		-1 & 2 & -2\\
		-1 & -2 & 2\\
		\end{pmatrix}$ & $\dfrac{K_2(2,F)}{\langle \{u^2,v\} \,\vert \, u,v\in F^{*}\rangle}$ \\
		\hline
		67. & $A=\begin{pmatrix}
		2 & -2 & -4\\
		-1 & 2 & -2\\
		-1 & -2 & 2\\
		\end{pmatrix}$ & $\dfrac{K_2(2,F)}{\langle \{u^2,v\} \,\vert \, u,v\in F^{*}\rangle}\times \dfrac{K_2(2,F)}{\langle \{u^2,v\} \,\vert \, u,v\in F^{*}\rangle}$ \\
		\hline
		81. & $A=\begin{pmatrix}
		2 & -2 & -3\\
		-2 & 2 & -2\\
		-1 & -2 & 2\\
		\end{pmatrix}$ & $\dfrac{K_2(F)}{4K_2(F)}\times \dfrac{K_2(2,F)}{\langle \{u^2,v\} \,\vert \, u,v\in F^{*}\rangle}$ \\
		\hline
		82. & $A=\begin{pmatrix}
		2 & -2 & -4\\
		-2 & 2 & -2\\
		-1 & -2 & 2\\
		\end{pmatrix}$ & $\dfrac{K_2(2,F)}{3\langle \{u^2,v\} \,\vert \, u,v\in F^{*}\rangle}\times \dfrac{K_2(2,F)}{\langle \{u^2,v\} \,\vert \, u,v\in F^{*}\rangle}$ \\
		\hline
		88. & $A=\begin{pmatrix}
		2 & -2 & -1\\
		-2 & 2 & -1\\
		-4 & -3 & 2\\
		\end{pmatrix}$ & $\dfrac{K_2(2,F)}{\langle \{u^2,v\} \,\vert \, u,v\in F^{*}\rangle}$ \\
		\hline
		90. & $A=\begin{pmatrix}
		2 & -1 & -2\\
		-4 & 2 & -4\\
		-2 & -1 & 2\\
		\end{pmatrix}$ & $K_2(2,F)\times \dfrac{K_2(2,F)}{\langle \{u^2,v\} \,\vert \, u,v\in F^{*}\rangle}$ \\
		\hline
	
	\end{tabular}
\end{center}
\begin{center}
	\begin{tabular}{ | m{1.5cm} | m{10em}| m{25em} | } 
		
		\hline
		Number. & Matrix $A$ & $K_2(A,F)$\\
		\hline
		103. & $A=\begin{pmatrix}
		2 & -2 & 0\\
		-2 & 2 & -1\\
		0 & -1 & 2\\
		\end{pmatrix}$ & $K_2(F)\times \dfrac{K_2(2,F)}{\langle \{u^2,v\} \,\vert \, u,v\in F^{*}\rangle}$ \\
		\hline
		106. & $A=\begin{pmatrix}
		2 & -2 & 0\\
		-2 & 2 & -2\\
		0 & -1 & 2\\
		\end{pmatrix}$ & $K_2(2,F)\times \dfrac{K_2(2,F)}{\langle \{u^2,v\} \,\vert \, u,v\in F^{*}\rangle}$ \\
		\hline
		110. & $A=\begin{pmatrix}
		2 & -1 & 0\\
		-4 & 2 & -2\\
		0 & -1 & 2\\
		\end{pmatrix}$ & $K_2(2,F)\times \dfrac{K_2(2,F)}{\langle \{u^2,v\} \,\vert \, u,v\in F^{*}\rangle}$ \\
		\hline
		113. & $A=\begin{pmatrix}
		2 & -2 & 0\\
		-2 & 2 & -1\\
		0 & -3 & 2\\
		\end{pmatrix}$ & $K_2(F)\times \dfrac{K_2(2,F)}{\langle \{u^2,v\} \,\vert \, u,v\in F^{*}\rangle}$ \\
		\hline
		114. & $A=\begin{pmatrix}
		2 & -2 & 0\\
		-2 & 2 & -3\\
		0 & -1 & 2\\
		\end{pmatrix}$ & $K_2(F)\times \dfrac{K_2(2,F)}{\langle \{u^2,v\} \,\vert \, u,v\in F^{*}\rangle}$ \\
		\hline
		116. & $A=\begin{pmatrix}
		2 & -1 & 0\\
		-4 & 2 & -2\\
		0 & -2 & 2\\
		\end{pmatrix}$ & $K_2(2,F)\times \dfrac{K_2(2,F)}{\langle \{u^2,v\} \,\vert \, u,v\in F^{*}\rangle}$ \\
		\hline
		119. & $A=\begin{pmatrix}
		2 & -1 & 0\\
		-3 & 2 & -4\\
		0 & -1 & 2\\
		\end{pmatrix}$ & $K_2(2,F)$ \\
		\hline
		122. & $A=\begin{pmatrix}
		2 & -1 & 0\\
		-4 & 2 & -4\\
		0 & -1 & 2\\
		\end{pmatrix}$ & $K_2(2,F)\times \dfrac{K_2(2,F)}{2\langle \{u^2,v\} \,\vert \, u,v\in F^{*}\rangle}$ \\ 
		\hline
	\end{tabular}
\end{center}
\end{adjustwidth}

\section{Rank 4 and Higher Hyperbolic GCMs}

Again, we use \cite{CHDD} to identify the hyperbolic GCMS. As before, we know how to compute $K_2(A,F)$ when $A$ is simply laced, or has 
an odd entry in each column, or in most cases when it has a column containing a -1 as its only non-zero off-diagonal entry. A quick survey of the table in \cite{CHDD} tells us that these techniques are
in fact enough to compute $K_2(A,F)$ (or at least reduce it to a smaller (affine or finite) case) for all of the rank 7 or higher GCMs. Furthermore, there is only two rank 5 
GCMs (No. 179 and No. 197) and one rank 6 GCM (No. 215) which can't be computed this way.

For the rank 4 GCMs, there are 11 (out of 52) which cannot be computed using these techniques. We present $K_2(A,F)$ for all these remaining GCMs in a table as we did for some of the rank 3 cases, and they were computed in a similar way.

\begin{adjustwidth}{-4em}{}
\begin{center}
	\begin{tabular}{ | m{1.5cm} | m{12em}| m{25em} | } 
		\hline
		Number. & Matrix $A$ & $K_2(A,F)$\\
		\hline
		128. & $A=\begin{pmatrix}
		2 & -1 & -1 & 0\\
		-1 & 2 & -1 & 0\\
		-1 & -1 & 2 & -2\\
		0 & 0 & -1 & 2 \\
		\end{pmatrix}$ & $K_2(2,F)$ \\
		\hline
		135. & $A=\begin{pmatrix}
		2 & -1 & 0 & -1\\
		-1 & 2 & -2 & 0\\
		0 & -1 & 2 & -1\\
		-1 & 0 & -2 & 2 \\
		\end{pmatrix}$ & $K_2(2,F)$ \\
		\hline
		143. & $A=\begin{pmatrix}
		2 & -1 & 0 & -2\\
		-1 & 2 & -1 & 0\\
		0 & -2 & 2 & -2\\
		-1 & 0 & -1 & 2 \\
		\end{pmatrix}$ & $\dfrac{K_2(2,F)}{\langle \{u^2,v\} \,\vert \, u,v\in F^{*}\rangle}$ \\
		\hline
		144. & $A=\begin{pmatrix}
		2 & -1 & 0 & -2\\
		-1 & 2 & -2 & 0\\
		0 & -1 & 2 & -2\\
		-1 & 0 & -1 & 2 \\
		\end{pmatrix}$ & $\dfrac{K_2(2,F)}{\langle \{u^2,v\} \,\vert \, u,v\in F^{*}\rangle}$ \\
		\hline
		146. & $A=\begin{pmatrix}
		2 & -1 & 0 & -2\\
		-2 & 2 & -2 & 0\\
		0 & -1 & 2 & -2\\
		-1 & 0 & -1 & 2 \\
		\end{pmatrix}$ & $K_2(2,F)$ \\
		\hline
		147. & $A=\begin{pmatrix}
		2 & -1 & 0 & -2\\
		-2 & 2 & -1 & 0\\
		0 & -2 & 2 & -2\\
		-1 & 0 & -1 & 2 \\
		\end{pmatrix}$ & $\dfrac{K_2(2,F)}{3\langle \{u^2,v\} \,\vert \, u,v\in F^{*}\rangle}$ \\
		\hline
		148. & $A=\begin{pmatrix}
		2 & -2 & 0 & -2\\
		-1 & 2 & -1 & 0\\
		0 & -2 & 2 & -2\\
		-1 & 0 & -1 & 2 \\
		\end{pmatrix}$ & $K_2(2,F)\times \dfrac{K_2(2,F)}{\langle \{u^2,v\} \,\vert \, u,v\in F^{*}\rangle}$ \\
		\hline
		156. & $A=\begin{pmatrix}
		2 & 0 & 0 & -1\\
		0 & 2 & 0 & -1\\
		0 & 0 & 2 & -1\\
		-2 & -2 & -2 & 2 \\
		\end{pmatrix}$ & $K_2(2,F)\times \dfrac{K_2(2,F)}{\langle \{u^2,v\} \,\vert \, u,v\in F^{*}\rangle}\times \dfrac{K_2(2,F)}{\langle \{u^2,v\} \,\vert \, u,v\in F^{*}\rangle}$ \\
		\hline
		162. & $A=\begin{pmatrix}
		2 & -1 & 0 & 0\\
		-1 & 2 & -1 & 0\\
		0 & -2 & 2 & -2\\
		0 & 0 & -1 & 2 \\
		\end{pmatrix}$ & $K_2(F)\times \dfrac{K_2(2,F)}{\langle \{u^2,v\} \,\vert \, u,v\in F^{*}\rangle}$\\
		\hline
		168. & $A=\begin{pmatrix}
		2 & -1 & 0 & 0\\
		-2 & 2 & -1 & 0\\
		0 & -1 & 2 & -3\\
		0 & 0 & -1 & 2 \\
		\end{pmatrix}$ & $K_2(2,F)$ \\
		\hline
	\end{tabular}
\end{center}
	\begin{center}
		\begin{tabular}{ | m{1.5cm} | m{12em}| m{25em} | } 
			\hline
			Number. & Matrix $A$ & $K_2(A,F)$\\
			\hline
			174. & $A=\begin{pmatrix}
			2 & -1 & 0 & 0\\
			-2 & 2 & -2 & 0\\
			0 & -1 & 2 & -2\\
			0 & 0 & -1 & 2 \\
			\end{pmatrix}$ & $K_2(2,F)\times \dfrac{K_2(2,F)}{\langle \{u^2,v\} \,\vert \, u,v\in F^{*}\rangle}$ \\
			\hline
	\end{tabular}
\end{center}	
\begin{center}
	\begin{tabular}{ | m{1.5cm} | m{17em}| m{20em} | } 
		\hline
		Number. & Matrix $A$ & $K_2(A,F)$\\
		\hline
		179. & $A=\begin{pmatrix}
		2 & -1 & -1 & 0 & 0\\
		-1 & 2 & 0 & -1 & 0\\
		-1 & 0 & 2 & -1 & 0\\
		0 & -1 & -1 & 2 & -2\\
		0 & 0 & 0 & -1 & 2\\
		\end{pmatrix}$ & $K_2(2,F)$\\
		\hline
		197. & $A=\begin{pmatrix}
		2 & -1 & 0 & 0 & 0\\
		-1 & 2 & -1 & 0 & 0\\
		0 & -2 & 2 & -1 & 0\\
		0 & 0 & -1 & 2 & -2\\
		0 & 0 & 0 & -1 & 2\\
		\end{pmatrix}$ & $K_2(F)\times \dfrac{K_2(2,F)}{\langle \{u^2,v\} \,\vert \, u,v\in F^{*}\rangle}$\\
		\hline
	\end{tabular}
\end{center}	
\begin{center}
	\begin{tabular}{ | m{1.5cm} | m{17em}| m{20em} | } 
		\hline
		Number. & Matrix $A$ & $K_2(A,F)$\\
		\hline
		215. & $A=\begin{pmatrix}
		2 & -1 & 0 & 0 & 0 & 0\\
		-1 & 2 & -1 & 0 & 0 & 0\\
		0 & -2 & 2 & -1 & 0 & 0\\
		0 & 0 & -1 & 2 & -1 & 0\\
		0 & 0 & 0 & -1 & 2 & -2\\
		0 & 0 & 0 & 0 & -1 & 2\\
		\end{pmatrix}$ & $K_2(F)\times \dfrac{K_2(2,F)}{\langle \{u^2,v\} \,\vert \, u,v\in F^{*}\rangle}$ \\
		\hline
	\end{tabular}
\end{center}
\end{adjustwidth}

\section{Arbitrary GCMs}

Throughout this paper, we have been focusing on the hyperbolic GCMs. However, the vast majority of indefinite GCMs are not of this type, so one may naturally ask how to go about computing $K_2(A,F)$ for an arbitrary GCM $A$. In this section, we shall survey the aspects of this paper which can be applied or generalised to this broader case.

The first point worth making is that in Sections 2--5 of this paper the hyperbolicity of the GCM is not used. Hence, the results proved in these sections can be applied equally well to any GCM. Furthermore, since every rank 2 GCM is finite, affine or hyperbolic, there is no specific requirement of hyperbolicity in Section 6 either.

In Sections 7 and 8 we used the classification of the hyperbolic GCMs as in \cite{CHDD} in order to reduce the number of remaining GCMs we wanted to consider to 60, many of which could be dealt with by hand or grouped together into similar cases. This is where the assumption of hyperbolicity was used most substantially.

So what benefits did the assumption of hyperbolicity provide us? The answer is two-fold. Firstly, the majority of hyperbolic GCMs could be handled using the methods of Sections 2--5. Secondly, those which couldn't tended to be of low rank or had a significant number of 0's and -1's as entries. This enabled us to apply the idea of Proposition \ref{del} in order to reduce $K_2(A,F)$ to a more manageable form - in particular, for almost all of the remaining GCMs we could reduce to $$K_2(A,F)=\frac{L_1\times L_2}{J}$$ where $L_1,L_2\in\{K_2(F),K_2(2,F)\}$ and $J=\langle \{u^a,v\}_1\{u^b,v\}_2^{-1}, \{u^{c_i},v\}_i\,\vert\,u,v\in F^{*}, i=1,2\rangle$ for some $a,b,c_1,c_2\in\B{Z}$. This form allowed us to use a slight variation of the proofs in Section 6 in order to construct isomorphisms to write $K_2(A,F)$ in a simpler way. (In fact, the only GCMs where we couldn't reduce to this form we could reduce to something that was clearly gave the same $K_2(A,F)$ as in Class 3 in Section 7).

For an arbitrary GCM we may still be able to apply the ideas of Proposition \ref{del} but the reduced form will in general remain somewhat complicated. It may be possible to construct an explicit isomorphism in these cases, if one has an idea of what the simplified presentation might look like.

This lead us to the following conjecture.

\begin{conjecture}
	Suppose that $A$ is an $n\times n$ GCM with the columns ordered such that the first $k$ columns all contain an odd entry, and the remaining columns have all even entries, for $k\in\{0,1,\ldots,n\}$. Let G be the abelian group defined by $G=\langle x_1,\ldots,x_n\,\vert\, x_i^{a_{ji}}=x_j^{a_{ij}}, [x_i,x_j]=1\; \mbox{for}\; 1\leq i<j\leq n\rangle$.
	
	Suppose further that $G\cong \langle y_1,\ldots,y_n\,\vert\,y_i^{r_i}=1\, [y_i,y_j]=1 \mbox{for}\,1\leq i\leq j\leq n\rangle$, where $r_1,\ldots,r_n\in\B{N}\cup\{0\}$ and the isomorphism is denoted by $\phi$ with inverse $\psi$.
	
	Let $\phi(x_i)=y_1^{\mu_{1i}}\ldots y_n^{\mu_{ni}}$ for $1\leq i\leq n$ and
	$\psi(y_i)=x_1^{\nu_{1i}}\ldots x_n^{\nu_{ni}}$ for $1\leq i\leq n$, where $\mu_{ji},\nu_{ji}\in\B{Z}$.
	
	Define $$L_i=\twopartdef  {K_2(F)} {\nu_{ji}\, \mbox{is even for all}\,\,\, k<j\leq n \,\,\,
		\mbox{and}\,\,\mu_{ij}\, \mbox{is odd for some}\,\,\, 1\leq j\leq k} {K_2(2,F)} {\mbox{not}}$$
	
	Define $$J_i=\twopartdef {r_iK_2(F)} { L_i=K_2(F)} {\frac{r_i}{2}\langle\{u^2,v\}\,\vert\,u,v\in F^{*}\rangle} { L_i=K_2(2,F)}$$
	
	Then $$K_2(A,F)\cong \frac{L_1}{J_1}\times\ldots\times\frac{L_n}{J_n}$$
\end{conjecture}

Various technical issues appear when trying to prove this conjecture in the same manner as Theorem \ref{odd}, which we have thus far been unable to overcome - in particular, it appears to require a detailed understanding of how the Fundamental Theorem of Finitely Generated Abelian Groups would give isomorphisms when $G$ is of the above form. However, given a specific GCM $A$ it should be possible to check this conjecture directly - we suggest that the correct isomorphism is $\Phi:K_2(A,F)\rightarrow \frac{L_1}{J_1}\times\ldots\times\frac{L_n}{J_n}$ defined by
$$\Phi(\{u,v\}_i)=[u^{\lvert\mu_{1i}\rvert},v]_1^{sign(\mu_{1i})}\ldots[u^{\lvert\mu_{ni}\rvert},v]_n^{sign(\mu_{ni})}$$ 

for $1\leq i\leq n$, where by $[.,.]_i$ we mean the Steinberg symbol/cocycle for the ith component of the direct product, and that its inverse is $\Psi:\frac{L_1}{J_1}\times\ldots\times\frac{L_n}{J_n}\rightarrow K_2(A,F)$ defined by 
$$\Psi([u,v]_i)=\{u^{\lvert\nu_{1i}\rvert},v\}_1^{sign(\nu_{1i})}\ldots\{u^{\lvert\nu_{ni}\rvert},v\}_n^{sign(\nu_{ni})}$$

If this conjecture - or a slightly modified one - is correct, this would allow us to simplify $K_2(A,F)$ for any GCM $A$, not just the hyperbolic ones.
\section{Acknowledgements}

This paper was written as part of the MA4K9 Project module during the 4th year of my undergraduate Mathematics degree at the University of Warwick. The supervisor for this project was Dr. Dmitriy Rumynin, to whom I give immense thanks for suggesting this project to me and for his continued assistance in its execution. I would also like to thank Dr. Inna Capdeboscq for some useful comments made regarding this paper.

This research did not receive any specific grant from funding agencies in the public, commercial, or not-for-profit sectors.

\section{Appendix}
Through the various chapters in this paper, we have been able to compute $K_2(A,F)$ for all the hyperbolic GCMs $A$. The results are summarised in the following table.

\begin{adjustwidth}{-5em}{}
\begin{center}
	\begin{tabular}{ | m{10em} | m{15em}| m{17em} | } 
		\hline
		Description of $A$. & Notation & $K_2(A,F)$\\
		\hline
		$A=\begin{pmatrix}
		2 & -b \\
		-a & 2\\
		\end{pmatrix}$ & $a,b$ even, $h=hcf(\frac{a}{2},\frac{b}{2})$ & $K_2(2,F)\times \dfrac{K_2(2,F)}{h\langle \{u^2,v\} \,\vert \, u,v\in F^{*}\rangle}$ \\
		\hline
		$A=\begin{pmatrix}
		2 & -b \\
		-a & 2\\
		\end{pmatrix}$ & $a$ even, $b$ odd, $h=hcf(a,b)$ & $K_2(2,F)\times \dfrac{K_2(F)}{hK_2(F)}$ \\
		\hline
		$A=\begin{pmatrix}
		2 & -b \\
		-a & 2\\
		\end{pmatrix}$ & $a,b$ odd, $h=hcf(a,b)$
		& $K_2(F)\times \dfrac{K_2(F)}{hK_2(F)}$ \\
		\hline
		$A$ simply-laced & - & $K_2(F)$ \\
		\hline
		All columns of $A$ have an odd entry & $\langle x_1,\ldots,x_n \,\vert\;\; x_i^{a_{ji}}=x_j^{a_{ij}}, \mbox{abel.}\rangle\cong \B{Z}/r_1\B{Z}\times\ldots\B{Z}/r_s\B{Z}\times \B{Z}^{n-s}$ & $\dfrac{K_2(F)}{r_1K_2(F)}\times\ldots\times \dfrac{K_2(F)}{r_sK_2(F)}\times K_2(F)^{n-s}$ \\
		\hline
		$a_{st}=-1$ only non-zero off-diagonal entry in $t$-th column, $a_{ts}$  not only odd entry in $s$-th column  & $A'$ finite or affine GCM obtained from $A$ by deleting $t$-th row and column & $K_2(A',F)$ \\
		\hline
		$A=\begin{pmatrix}
		2 & -a & -b\\
		-c & 2 & -d\\
		-1 & -1 & 2\\
		\end{pmatrix}$ & $b,d$ even & $\dfrac{K_2(2,F)}{(\frac{ad-bc}{2})\langle \{u^2,v\} \,\vert \, u,v\in F^{*}\rangle}$ \\
		\hline
		$A=\begin{pmatrix}
		2 & -1 & -a\\
		-b & 2 & -c\\
		-1 & -d & 2\\
		\end{pmatrix}$ & $a,c$ even & $\dfrac{K_2(2,F)}{(\frac{abd-c}{2})\langle \{u^2,v\} \,\vert \, u,v\in F^{*}\rangle}$ \\
		\hline
		Remaining hyperbolic $A$ & - & Can be found in tables in Chapters 7 and 8 \\
		\hline
	\end{tabular}
\end{center}
\end{adjustwidth}

\section{References}
\begin{thebibliography}{99}
\bibitem{LAFAT} Carter, R., \emph{Lie Algebras of Finite and Affine Type}, Cambridge University Press 2005
\bibitem{IDLA} Kac, V., \emph{Infinite Dimensional Lie Algebras}, Cambridge University Press 1990, Third Edition
\bibitem{IAKT} Milnor, J., \emph{Introduction to Algebraic K-Theory}, Annals of Mathematics Studies, Princeton University Press 1971
\bibitem{AMTTFKMG} Morita, J. and Rehmann, U., \emph{A Matsumoto-type theorem for Kac-Moody groups}, Tohoku Math. J. (2) 42 (1990) no. 4, 537--560.
\bibitem{SLOCG} Steinberg, R., \emph{Lectures on Chevalley groups}, Yale Univ. Lecture Notes, New Haven CT, 1968
\bibitem{MAT} Matsumoto, H., \emph{Sur les sous-groupes arithm\'etiques des groupes semi-simples d\'eploy\'es}, Ann. Scient.
Ec. Norm. Sup. (4)2 (1969), 1--62
\bibitem{CHDD}Carbone, L., Chung, S., Cobbs, L., Mcrae, R., Nandi, D., Naqvi, Y., Penta, D., \emph{Classification of
hyperbolic Dynkin diagrams, root lengths and Weyl group orbits}, J. Phys. A: Math. Theor.,
43, No.15, 2010, 155209 (30 pp), doi:10.1088/1751-8113/43/15/155209
\bibitem{TIKF} Suslin, A.A., \emph{Torsion in $K_2$ of fields}, K-Theory 1 (1987), 5--29
\bibitem{UPKMG} Tits, J., \emph{Uniqueness and presentation of Kac-Moody groups over fields}, J. Algebra 105 (1987), 542--573
\bibitem{DVR} Dennis, R. and Stein, M., \emph{$K_2$ of discrete valuation rings}, Advances in Math. Vol. 18 (1975), 182--238 
\end{thebibliography}
\end{document}